\documentclass{amsart}

\usepackage{amsfonts,amsmath,amsthm,comment,graphicx,amssymb, url}

\newtheorem{definition}{Definition}
\newtheorem{theorem}{Theorem}
\newtheorem{cor}{Corollary}
\newtheorem{remark}{Remark}
\newtheorem{lemma}{Lemma}
\newtheorem{question}{Question}
\newtheorem{corollary}{Corollary}

\newcommand{\acts}{\curvearrowright}
\newcommand{\N}{\mathcal{N}}
\newcommand{\St}{\text{St}}

\title{Universal minimal flows of extensions of and by compact groups}
\author[D. Barto\v{s}ov\'{a}]{Dana Barto\v{s}ov\'{a}}
\address{Department of Mathematics,
	University of Florida,
	PO Box 118105,
	Gainesville, FL 32611.}
\email{dbartosova@ufl.edu}
\thanks{This research was partially supported by NSF grant DMS-1953955.}

\begin{document}
	\maketitle
	
	\begin{abstract}
		Every topological group $G$ has up to isomorphism a unique minimal $G$-flow that maps onto every minimal $G$-flow, the universal minimal flow $M(G).$ We show that if $G$ has a 
		compact  normal subgroup $K$ that acts freely on $M(G)$ and 
		there exists a uniformly continuous cross section $G/K\to G,$ then the phase space of $M(G)$ is homeomorphic to the product of the phase space of $M(G/K)$ with $K$. Moreover, if either the left and right uniformities on $G$ coincide or $G\cong G/K\ltimes K$, we also recover the action, in the latter case extending a result of Kechris and Soki\'c. As an application, we show that the phase space of $M(G)$ for any totally disconnected locally compact Polish group $G$ with a normal open compact subgroup is homeomorphic to a finite set, Cantor set $2^{\mathbb{N}}$, $M(\mathbb{Z})$, or $M(\mathbb{Z})\times 2^{\mathbb{N}}.$
	\end{abstract}
	
	\section{Introduction}
	Structure theory of minimal flows dates to the origins of topological dynamics. Among the most notable achievements is the structure theorem on minimal metrizable distal flows of countable discrete groups by Furstenberg (\cite{F}). Furstenberg's result was extended to non-metrizable minimal distal flows by Ellis in \cite{E} and to locally compact groups by Zimmer in \cite{Z} and \cite{Z2}. 
	However, little is known about the structure of general minimal flows. A well-known theorem by Ellis states that up to isomorphism, every topological group $G$ admits a unique universal minimal flow $M(G)$ (mapping onto every minimal flow preserving the respective actions). 
		Using Boolean theoretic methods, Balcar and Fran\v{e}k (\cite{BF}) showed that the underlying space (phase space) of $M(G)$ of any countably infinite discrete group $G$ is homeomorphic to the Gleason cover $G(2^{\kappa})$ of the Cantor cube $2^{\kappa}$ (unique compact extremely disconnected space that has an irreducible map onto $2^{\kappa}$, i.e., $2^{\kappa}$ is not the image of any proper clopen subset of $G(2^{\kappa})$) for some uncountable cardinal $\kappa\leq 2^{\aleph_0}$. It was proved that $\kappa=2^{\aleph_0}$ for Abelian groups by Turek  (\cite{T}) and in general, by a recent paper of Glasner, Tsankov, Weiss, and Zucker (\cite{GTWZ}). 
	
	\begin{theorem}[Balcar and Fran\v{e}k (\cite{BF}), Turek (\cite{T}), Glasner, Tsankov, Weiss, and Zucker (\cite{GTWZ})]\label{discrete}
		Let $G$ be a discrete countably infinite group. Then the phase space of $M(G)$ is homeomorphic to the Gleason cover of the Cantor cube $2^{2^{\aleph_0}}.$
	\end{theorem}

Up-to-date, we lack a concrete description of the universal minimal action  even in the case of the discrete group of integers $\mathbb{Z}$. 
	
	
	On the other hand, there are topological groups for which the universal minimal flow (and therefore any minimal flow) is metrizable (e.g., the group of permutations of a countable set $S_{\infty}$ (\cite{GW}), the group of homeomorphisms of the Lelek fan (\cite{BK}) ), or even trivializes (e.g., the group of automorphisms of the rational linear order (\cite{P}), group of unitaries of the separable Hilbert space (\cite{GM})). These results are obtained via structural Ramsey theory (\cite{P},\cite{KPT}, \cite{GW}) or concentration of measure phenomena (\cite{GM}, \cite{P2}), and provide a concrete description. If the group $G$ can be represented as a group of automorphisms of a countable first-order structure, then the phase space of $M(G)$ is either finite or homeomorphic to the Cantor set $2^{\mathbb{N}}.$ However, we lack group topological characterization of such groups. 
	
	Generalizing Theorem \ref{discrete}, Bandlow  showed  in \cite{Ba} that if  an  $\omega$-bounded group $G$ (every maximal system of pairwise disjoint translates of a neighbourhood in $G$ is countable) has an infinite minimal flow, then the phase space of $M(G)$  has the same Gleason cover as $2^{\kappa}$ for some infinite cardinal $\kappa$ (for a  simplified proof see \cite{T2}). B\l aszczyk, Kucharski, and Turek  demonstrated in \cite{BKT} that every infinite minimal flow of such a group maps irreducibly onto $2^{\kappa}$ for some infinite $\kappa$.
	
	Our original motivation was to prove an analogue of Theorem \ref{discrete} for Polish (separable and completely metrizable) locally compact groups with a basis at the neutral element of open subgroups. They are a natural direction from discrete groups since their topology is determined by cosets of basic open (compact) subgroups and their universal minimal flow can be viewed from the Boolean algebraic perspective. This class of groups coincides with the class of locally compact groups of automorphisms of countable first order structures, which in turn by van Dantzig's Theorem (\cite{vD}) coincides with the class of Polish totally disconnected locally compact (t.d.l.c.) groups. We succeed in case $G$ possesses a compact open normal subgroup.

	\begin{theorem}
		Let $G$ be a Polish t.d.l.c. group with a compact normal open subgroup $K.$ Then $M(G)$ is homeomorphic to a finite set, $M(\mathbb{Z}),$ $2^{\mathbb{N}},$ or $M(\mathbb{Z})\times 2^{\mathbb{N}}$.
	\end{theorem} 

	If $G$ admits a basis of compact open normal subgroups then we also recover the action. Among Polish groups, these are the t.d.l.c. groups admitting a two sided invariant metric (TSI), in particular, Abelian groups.

	These are consequences of the main results of the present paper, that can be summarized as follows. 
	
	\begin{theorem}\label{combined}
		Let $G$ be a topological group with a compact normal subgroup $K$ such that $K$ acts freely on $M(G)$ and the quotient mapping $G\to G/K$ admits a uniformly continuous cross section. Then the phase space of $M(G)$ is homeomorphic to the product of the phase space of $M(G/K)$ and $K.$ If the left and right uniformities on $G$ coincide or the cross section is a group homomorphism, then the homeomorphism is an isomorphism of flows. 
	\end{theorem}

As a corollary, we obtain a result by Kechris and Soki\'c.

\begin{corollary}[\cite{KS}]
	Let $G$ be a Polish group with a metrizable $M(G)$ and let $K$ be a compact metrizable group. Suppose that $G$ acts continuously on $K$ by automorphisms. Then $M(G\ltimes K)$ is isomorphic to $M(G)\times K.$
\end{corollary}

In Section \ref{by}, we verify that an analogue of Theorem \ref{combined} holds when $K$ is closed (not necessarily compact), but $G/K$ is compact, extending another result of Kechris and Soki\'c from \cite{KS}.

The strategy is to prove statements in Theorem \ref{combined} for the greatest ambit of $G$ that contains $M(G)$ as its minimal subflow.

	
	
	Let us remark that recently Basso and Zucker isolated  in  \cite{BZ} a class of topological groups for which $M(G\times H)\cong M(G)\times M(H),$ containing all groups with metrizable universal minimal flows. On the contrary, for any two infinite discrete group $G,H$, $M(G\times H)\not\cong M(G)\times M(H),$ in fact, their phase spaces are not even homeomorphic. 
	
	

	\section{Background}
	Let $G$ be a topological group with neutral element $e$. The topology on $G$ is fully determined by a basis at $e$ of open neirghbourhoods, that will be denoted by $\N_e(G).$ Withouth loss of generality, we can assume that $V=V^{-1}$ for every $V\in\N_e(G).$ 
	
	A \emph{$G$-flow} is a continuous left action $\alpha:G\times X\to X$ of $G$ on a compact Haursdorff space $X$, to which we refer to as the \emph{phase space} of $\alpha$. We typically write $gx$ in place of $\alpha(g,x).$ The action is \emph{free} if $G$ acts without fixed points on $X$, i.e.,  for every $g\neq e$ and $x\in X,$ $gx\neq x.$ A \emph{homomorphism} between  two $G$-flows $G\times X\to X$ and $G\times Y\to Y$ is a continuous map $\phi:X\to Y$ such that $\phi(gx)=g\phi(x).$  An \emph{isomorphism} is a bijective homomorphism (recall that a bijective continuous map between compact spaces is a homeomorphism). We will use $\cong$ for the isomorphism relation on flows. An \emph{ambit} is a $G$-flow $X$ with a \emph{base point} $x_0\in X$ that has a dense orbit, i.e., $Gx_0=\{gx_0:g\in G\}$ is dense in $X$.
	An \emph{ambit homomorphism} is a homomorphism between ambits $(X,x_0)$ and $(Y,y_0)$ sending $x_0$ to $y_0.$ There is, up to isomorphism, a unique ambit $(S(G),e)$  that homomorphically maps onto every ambit, the \emph{greatest ambit}. Topologically, $S(G)$ is the \emph{Samuel compactification} of $G$ with respect to the \emph{right uniformity} on $G$ generated by open covers \[
	\{\{Vg:g\in G\}:V\in \N_e(G)\}.
	\]
	A compact space $X$ admit a unique compatible uniformity (generated by finite open covers). We will say that a map $G\to X$ is \emph{right uniformly continuous } if it is uniformly continuous with respect to the right uniformity on $G$ and the unique compatible uniformity on $X.$

	 The Samuel compactification is the smallest compactification of $G$ (i.e., there is an embedding of $G$ into $S(G)$ with dense image) such that every right uniformly continuous function from $G$ to a compact space (uniquely) extends to $S(G).$ Similarly, we define left uniformity, right action, and the greatest ambit with respect to the right action. We call a topological group \emph{SIN} if the left and right uniformities coincide. It is easy to see that a group $G$ is SIN if and only if it admits a basis of $e$ consisting of conjugation invariant neighbouroods, that is, $V$ such that $gVg^{-1}=V$ for every $g\in G.$ Multiplication and inversion in SIN groups are (right, left) uniformly continuous. 
	 
	 We will describe Samuel's original construction (see \cite{S}) of the Samuel compactification for a topological group $G$ with respect to the right uniform structure. For a discrete group $H,$ the Samuel compactification coincides with the \v{C}ech-Stone compactification $\beta H.$ $\beta H$ consists of ultrafilters on $H$ with a basis for topology of clopen sets $\hat{A}=\{u\in \beta(H), A\in u\}$ for $A\subset H$. We can identify $H$ with a dense subset of $\beta H$ via principal ultrafilters. The action $H\times\beta H\to \beta H$ given by $hu=\{hA:A\in u\}$ is the greatest ambit $(\beta H,e).$ Given a topological group $G$, we denote by $G_d$ the same algebraic group with the discrete topology. Then the greatest ambit action of $G$ remains an ambit action $G_d\times (S(G),e)\to (S(G),e).$ Since $\beta G_d$ is the greatest ambit for $G_d,$ there is an ambit homomorphism $\pi:\beta G_d\to S(G).$ For an ultrafilter $u$ in $\beta G_d,$ we let the \emph{envelope} of $u$ be its subfilter
	 \[\label{envelope}
	 u^*=\left<
	 \{VA:V\in\N_e(G),A\in u\} \tag{$\Delta$}
	 \right>
	  \]
	  generated by $\{VA:V\in\N_e(G),A\in u\}.$ Given $u,v\in\beta G_d,$ we set $u\sim v$ iff $u^*=v^*$. Then $\sim$ is an equivalence relation whose equivalence classes are exactly preimages of $\pi.$ The collection of 
	  \[
	  \bar{A}=\{u^*\in S(G):u^*\supset\{VA:V\in\N_e(G)\}\}
	  \]
	  for $A\subset G$ forms a basis for closed sets in $S(G).$
	 
	 A $G$-flow on $X$ is \emph{minimal} if $X$ contains no closed proper non-empty invariant subset, i.e., no proper \emph{subflow}. Equivalently, for every $x\in X,$ the orbit $Gx$ is dense in $X$. Up to isomorphism, there is a unique minimal flow that admits a homomorphism onto every minimal flow, called \emph{the universal minimal flow} of $G$ and denoted by $M(G)$. The universal minimal flow is isomorphic to any minimal subflow of $S(G)$. If $G$ is compact, then $M(G)\cong G$ with the left translation action, if $G$ is locally compact, non-compact, then $G$ acts freely on $M(G)$, by a theorem of Veech (\cite{V}), and it is nonmetrizable (\cite{KPT}). 
	
	\section{Extensions of compact groups}

	Given topological groups $G,K,H,$ we say that $G$ is an \emph{extension} of $K$ by $H$ if there is a short exact sequence
	\[
	\{e\}\to K\to G\to H\to \{e\},
	\]
	 where each arrow is a continuous open group homomorphism onto its image. We will focus on the case that $K$ is compact. Identifying $K$ with its image, we can assume that $K$ is a compact normal subgroup of $G$ and $H$ is isomorphic to $G/K.$ For the rest of this section, we fix a topological group $G$ together with a compact normal subgroup $K.$ We will investigate the relationship between $S(G)$ and $S(G/K)\times K,$ respectively, $M(G)$ and $M(G/K)\times K.$
	 
	 Let $S(G)/K$ denote the orbit space $\{Kx:x\in S(G)\}$ of the restricted  action $K\times S(G)\to S(G)$. Since $K$ is compact, $K$-orbits are equivalence classes of a closed equivalence relation on  $S(G),$ and therefore $S(G)/K$ (with the quotient topology) is a compact Hausdorff space. We have that $G/K\times S(G)/K\to S(G)/K$ defined by $(Kg)Kx=Kgx$ is a continuous ambit action with the base point $K.$
	
	\begin{lemma}\label{orbits}
		$S(G/K)\cong S(G)/K.$
		\end{lemma}
	
	\begin{proof}
		We can define an ambit action $G\times S(G/K)\to S(G/K)$ by $gx=Kgx$. Since $S(G)$ is the greatest $G$-ambit, there is an ambit homomorphism $\phi:S(G)\to S(G/K)$. Clearly, $\phi$ is constant on every $K$-orbit, so $\phi$ factors through $\psi:S(G)\to S(G)/K$ and $\xi:S(G)/K\to S(G/K).$ Since $\xi$  maps $Ke\in S(G)/K$ to $K\in G/K\subset S(G/K)$, it is a $G/K$-ambit homomorphism.  
		 But $S(G/K)$ is the greatest $G/K$-ambit, therefore $\xi$ must be an ambit isomorphism.
	\end{proof}

\begin{remark}\label{nearult}
By Lemma \ref{orbits} and (\ref{envelope}), we have that for any $u\in\beta G_d,$ the $K$-orbit $Ku^*=\{ku^*:k\in K\}$ is a closed subset of $S(G)$ corresponding to the filter  \[\left<\{VKA:V\in\mathcal{N}_e(G),A\in u\}\right>.\] That is, elements of $Ku^*$ are exactly those $v^*$ that extend it. Observe, that this is not the case when $K$ is not compact, not even if we replace the orbits of $K$ with their closures. 
 \end{remark}

\begin{lemma}
	$M(G/K)\cong M(G)/K.$
\end{lemma}

\begin{proof}
	Since $M(G)$ is  isomorphic to a minimal $G$-subflow of $S(G),$ $M(G)/K$ is isomorphic to a minimal $G/K$-subflow of $S(G)/K.$ By Lemma \ref{orbits}, $S(G)/K$ is isomorphic to the greatest ambit $S(G/K)$ and therefore its minimal subflows are isomorphic to $M(G/K).$ 
	
	Let us remark that we can prove this fact directly as in the proof of Lemma \ref{orbits} from universality and uniqueness of $M(G).$
	\end{proof}

\begin{definition}
	Let $X$ be a topological space and let $\pi:X\to X/E$ be a quotient map. A cross-section is a  map $\phi:X/E \to X$ such that $\pi\circ\phi=\text{id}_{X/E}.$
\end{definition}

\begin{theorem}\label{section}
	Suppose that there exists a uniformly continuous cross section $s:G/K\to G.$ Then there is a continuous cross section $s':S(G)/K\to S(G)$. 

\end{theorem}

\begin{proof}
	Let $\pi:S(G)\to S(G/K)$ be the natural $G$-ambit homomorphism as in Lemma \ref{orbits}.
	Let $s:G/K\to G$ be a uniformly continuous cross section. Since  uniformly continuous maps from $G/K$ to compact spaces uniquely extend to $S(G/K),$ there is a continuous map $\bar{s}:S(G/K)\to S(G)$ extending $s$. Concretely (see \cite{B}), denoting by $S$ the image of $s,$ for any $u\in\beta G_d,$ \[\bar{s}(\pi(u^*))=\left<\{VKA:V\in\mathcal{N}_e(G), A\in u\}\cup\{S\}\right>^*.\]
	Since $S$ intersects every coset of $K$ and $\{KA:A\in u\}$ is an ultrafilter in $\beta(G/K)_d,$ we get that $v=\left<\{KA:A\in u\}\cup \{S\}\right>$ is an ultrafilter in $\beta G_d$ such that $\{KA:A\in u\}=\{KB:B\in v\}.$
	 By Remark \ref{nearult}, it follows that $\bar{s}\pi(u^*)$ lies in the orbit $Ku^*$. Precomposing with $\xi$ from Lemma \ref{orbits}, we obtain a continuous cross section $s':S(G)/K\to S(G).$
\end{proof}

\begin{remark}\label{continuous}
	Continuity of $s'$ means that for every open $U\subset S(G)$ there is an open $V\subset S(G)$ such that $s'(KV)\subset U$ from the definition of the quotient topology on $S(G)/K.$ 
\end{remark}


\begin{corollary}\label{product}
	Suppose that $K$ acts freely on $S(G).$
	If there exists a uniformly continuous cross section $s:G/K\to G,$ then the phase space of  $S(G)$ is homeomorphic to  $S(G/K)\times K.$
\end{corollary}

\begin{proof}
	Denote by $s':S(G)/K\to S(G)$ the continuous cross section as in Theorem \ref{section}.
	Define a map $\phi:S(G)/K\times K\to S(G)$ by $\phi(Ku,k)=ks'(Ku).$ As $K$ acts freely, $\phi$ is a bijection. Since $\phi$ is multiplication of two continuous maps, it is continuous, hence a homeomorphism.  The statement follows by Lemma \ref{orbits}.
\end{proof}

\begin{corollary}\label{main}
	If there exists a uniformly continuous cross section $s:G/K\to G,$ then the phase space of $M(G)$ is homeomorphic to $M(G/K)\times K.$
\end{corollary}

If $G$ is SIN, we can fully describe the greatest ambit and the universal minimal actions.

\begin{theorem}\label{flows}
	Let $G$ be SIN. Suppose that $K$ acts freely on $S(G)$ and that there is a uniformly continuous cross section $s:G/K\to G.$ Then there is an action of $G$ on $S(G/K)\times K$ such that $S(G)\cong S(G/K)\times K$ as ambits.

\end{theorem}

\begin{proof}
	
	
	Let $\N_e(G)$ be a neighbourhood basis of $G$ at $e$ consisting of $V$ such that $gVg^{-1}=V$ for every $g\in G.$ Then for every $A\subset G,$ $VA=AV,$ therefore the construction of $S(G)$ by Samuel via envelopes yields the same space. Similarly, since $K$ is normal in $G,$ $VKA=AKV,$ we have that the quotient maps $S(G)\to S(G/K)$ and $S(G)\to S(K\backslash G)$ define the same equivalence relation on $S(G).$ By Lemma \ref{orbits},  the orbit spaces $S(G)/K=\{Ku:u\in S(G)\}$ and $K\backslash S(G)=\{uK:u\in S(G)\}$ are the same. It means that for every $k\in K, u\in S(G)$ there is $k'\in K$ such that $ku=uk'.$ 
	
	Suppose that there is a uniformly continuous cross section $s:G/K$ to $K;$ by shifting we can assume that $s(K)=e.$ By Theorem \ref{section},  $s$ uniquely extends to a cross section $s':S(G)/K\to S(G)$.  
	
		We define $\rho: G\times S(G)/K\to K$ to ``correct'' that $s$ may not be a group homomorphism by the rule
	\[
	s'(Kgu)\rho(g, Ku)=gs'(Ku).
	\]
	Since left and right $K$-orbits in $S(G)$ coincide and $K$ acts freely on $S(G),$ $\rho$ is well defined. 
	
	We will prove that $\rho$ is continuous. Fix $g\in G$ and $Ku\in S(G)/K$, and denote $k=\rho(g, Ku).$ Let $k\in O$ be a basic open neighbourhood of $k$ in $K$. 
	By Corollary \ref{product} for right actions, we have that $ S(G)/K\times G\to S(G)$ defined by $(Ku,l)\mapsto s'(Ku)l$ is a homeomorphism. Consequently, $s'(S(G)/K)\times O$ is an open neighbourhood of $gs'(Ku)$ in $S(G).$ Since the action of $G$ on $S(G)$ is continuous, there are  open neighbourhoods $V\in \mathcal{N}_e(G)$ and $U$ of $s'(u)$ such that $VgU\subset s'(S(G)/K)\times O.$ Since $s'$ is continuous, there is an  open neighbourhood $U'$ of $u$ such that $s'(KU')\subset U.$
	Since  $U'$ is open in $S(G)$ and $G$ acts by homeomorphisms, $VgU'=\bigcup_{h\in Vg}hU',$ is open in $S(G).$ Therefore, $VgU'k^{-1}$ is an open neighbourhood of $s'(gKu).$ By continuity of $s'$ and of the action of $G$ on $G/K$, there are neighbourhoods $V'\in\mathcal{N}_e(G)$ and $U''$ of $u$ in $S(G)$ such that $s'(V'gKU'')\subset VgU'k^{-1}.$ Finally, we get that $s'((V\cap V')gK(U'\cap U''))\times O\supset (V\cap V')gs'(K(U'\cap U'')),$ so $\rho((V\cap V')g,K(U\cap U''))\subset O.$
	
	
	The function $\rho$ is a cocycle in a sense that for every $g,h\in G$ we have 
	\begin{align}\label{cocycle}
	\rho(gh,Ku)=\rho(g,Khu)\rho(h,Ku): \tag{\textasteriskcentered}
	\end{align}
	\begin{align*}
		s'(Kghu)\rho(gh, Ku)=&ghs'(Ku)\\
		s'(Kghu)\rho(g,Khu)=&gs'(Khu)\\
		s'(Khu)\rho(h,Ku)=&hs'(Ku)\\
		 s'(Kghu)\rho(gh, Ku)=&gs'(Khu)\rho(h,Ku)=s'(Kghu)\rho(g,Khu)\rho(h,Ku),
	\end{align*}
	and (\ref{cocycle}) follows from freeness of the action by $K.$ 
	
	We define an action $\alpha:G\times (S(G/K)\times K)\to S(G/K)\times K$ by 
	\[
	g(Ku,k)=(Kgu,\rho(g,Ku)k).
	\] 
	
	By (\ref{cocycle}), $\alpha$ is an action and it is obviously continuous. 
	
	We define \[\phi:S(G/K)\times K\to S(G), (Ku, k)\mapsto s'(Ku)k.\] Then $\phi$ is an ambit homomorphism: 
	\begin{enumerate}
		\item $\phi$ is continuous as it is a composition of multiplication with continuous functions. 
		\item $\phi(K,e)=s'(K)e=ee=e$.  
		\item  $\phi(g(Ku,k))=\phi((Kgu,\rho(g,Ku)k))=s'(Kgu)\rho(g,Ku)k=gs'(Ku)k=g(\phi(Ku,k))$, where the second to last equality holds  by the definition of $\rho.$
		\end{enumerate}
	
	By universality of $S(G),$ we can conclude that $\phi$ is an isomorphism.

\end{proof}

\begin{corollary}\label{iso}
		Let $G$ be SIN. Then there is an action of $G$ on $M(G/K)\times K$ such that $M(G)\cong M(G/K)\times K$.
\end{corollary}

\section{Extensions by compact groups}\label{by}
In this section we will consider short exact sequences
\[
\{e\}\to N\to G\to K\to \{e\},
\]
where $K$ is compact. We verify that a result of Kechris and Soki\'c for Polish $G$ in \cite{KS} generalizes to arbitrary topological groups and works for the greatest ambit as well as the universal minimal flow. Our proof is slightly shorter, but the idea is the same. 

\begin{theorem}\label{bythm}
Let $G$ be a topological group with a closed normal subgroup $N$ such that $G/N$ is compact. Suppose that there is a continuous cross section $s:G/N\to G.$  Then $S(G)\cong S(N)\times G/N$.
\end{theorem}

\begin{proof}
	
	By shifting, we can assume that $s(N)=e.$

We again define a cocycle $\rho: G\times G/N\to N$ by 
\[
s(Ngh)\rho(g, Nh)=gs(Nh).
\]

We have that $\rho(g, Nh)=s(Ngh)^{-1}gs(Nh),$ so $\rho$ is continuous. By (\ref{cocycle}) in the proof of Theorem \ref{flows}, $\rho(gg', Nh)=\rho(g, Ng'h)\rho(g',Nh).$ Therefore, the map $G\times (S(N)\times G/N)\to S(N)\times G/N$ defined by 
\[
g(u, Nh)=(\rho(g, Nh)u, Ngh)
\]
is an ambit action. 

Viewing $S(G)$ as an $N$-flow, there is an $N$-ambit homomorphism $\mu:(S(N),e)\to (S(G),e)$. 
We define $\phi:S(N)\times G/N \to S(G)$ by $(u, Nh)\mapsto s(Nh)\mu(u).$ Then $\phi$ is a $G$-ambit homomorphism:
\begin{enumerate}
	\item $\phi$ is continuous since, $s, \mu,$ and the action of $G$ on $S(G)$ are.
	\item $\phi(e,N)=s(N)\mu(e)=ee=e.$
	\item Since $\mu$ is a homomorphism of $N$-flows, we have $\mu(\rho(g,Nh)u)=\rho(g,Nh)\mu(u).$		
	\begin{align*}
		\phi(g(u,Nh))=&\phi(\rho(g,Nh)u, Ngh)=s(Ngh)\mu(\rho(g,Nh)u). \\
					 =& s(Ngh)\rho(g,Nh)\mu(u) \tag{$\mu$ homomorphism}\\
					 =&gs(Nh)\mu(u)= g\phi(u,Nh) \tag{$\rho$ cocycle} 	
		\end{align*}
		
	\end{enumerate}

Since $S(G)$ is the greatest ambit,  $\phi$ is an isomorphism. 
\end{proof}

\begin{cor}
Let $G$ be a topological group with a closed normal subgroup $N$ such that $G/N$ is compact. Suppose that there is a continuous cross section $s:G/N\to G.$  Then $M(G)\cong M(N)\times G/N$.	
\end{cor}

\begin{proof}
	For every $Nh\in G/N,$ the evaluation map $G\to N,$ $g\mapsto\rho(g, Nh) $ given by the cocycle $\rho$ defined in the proof of Theorem \ref{bythm} is onto. Viewing $M(N)$ as a subflow of $S(N)$, for any $m\in M(N)$, $\{\rho(g,Nh)m:g\in G\}=Nm$, so $M(N)\times G/N$ is a minimal subflow of $S(N)\times G/N.$
\end{proof}
	
	\section{Applications}
	\subsection{Totally disconnected locally compact groups}
	Totally disconnected locally compact (t.d.l.c.) groups coincide with locally compact groups of automorphisms of first order structures. T.d.l.c. groups  admit a local basis at the neutral element $e$ consisting of open compact subgroups. Systematic study of Polish t.d.l.c. groups was started by the book of Wesolek (\cite{W}). By van Dantzig's theorem, the underlying topological space of a Polish t.d.l.c. group is homeomorphic to either a countable set, the Cantor space $2^{\mathbb{N}},$  or to $\mathbb{N}\times 2^{\mathbb{N}}$. By Veech's theorem (\cite{V}), every locally compact group acts freely on its greatest ambit, and consequently on its universal minimal flow. If $K$ is a compact open normal subgroup of $G,$ then $G/K$ is a discrete group. Therefore, any cross section $G/K\to G$ is uniformly continuous and we can apply Corollary \ref{main} to derive the following.
	
	\begin{theorem}\label{tdlc}
		Let $G$ be a t.d.l.c. group with a normal compact open subgroup $K.$ Then the phase space of $M(G)$ is homeomorphic to $M(G/K)\times K.$
	\end{theorem}

In the case of Polish t.d.l.c. group we get a complete characterization of phase spaces of universal minimal flows. 

	\begin{corollary}\label{tdlciso}
		If $G$ is a Polish t.d.l.c. group with a normal compact open subgroup $K,$ then $M(G)$ is homeomorphic to a finite set, $M(\mathbb{Z}), 2^{\mathbb{N}}$,  or $M(\mathbb{Z})\times 2^{\mathbb{N}}.$
	\end{corollary}
	
	In case that $G$ in Theorem \ref{tdlc} is moreover SIN, we can also define the action as in Corollary \ref{iso}. 
	
	\begin{theorem}
		If $G$ is a SIN t.d.l.c. group and $K$  any normal compact open subgroup, then $M(G)$ is isomorphic  to $M(G/K)\times K.$
	\end{theorem}

Polish SIN t.d.l.c. groups are exactly locally compact groups of automorphisms of countable first order structures admitting a two sided invariant metric (TSI), equivalently, Polish groups admitting a countable basis at the neutral element consisting of compact open normal subgroups.

	\subsection{Semi-direct products}
	
	In \cite{KS}, Kechris and Soki\'c studied universal minimal flows of semidirect products of Polish groups with one of the factors compact. Their proof can be modified to apply to groups that are not Polish as well. We include the results for completeness and extend them to the greatest ambit. 
	
	In Theorem \ref{flows} we required $G$ to be SIN and in Theorem \ref{bythm}, $S(G/N)=G/N,$ which allowed us to define the cocycle $\rho$ to compensate for $s$ not being necessarily a group homomorphism. If $s$ is a group homomorphism, that is, the respective short exact sequence splits, then $G$ is a semidirect product, and we do not need $\rho$.
	
	\begin{theorem}
		Let $G\cong H\ltimes K,$ where $K$ is compact. Then $S(G)\cong S(H)\times K$ and $M(G)\cong M(H)\times K.$
		\end{theorem} 
	
	\begin{proof}
		Let $s:H\to G$ be a cross section that is a continuous homomorphism and let $\pi:G\to H$ be the natural projection. Define $\alpha: G\times (S(H)\times K)\to (S(H)\times K)$ by $g(u,k)=(\pi(g)u,gks(\pi(g))^{-1}).$ Since $\pi$ and $s$ are continuous group homomorphisms, $\alpha$ is a continuous action. Also, $G(e,e)=(He, K)$ is dense in $S(H)\times K,$ so $\alpha$ is an ambit action. 
		Since $s$ is a continuous homomorphism, we can define a continuous ambit action  $H\times S(G)\to S(G)$ by $(h,u)\mapsto s(h)u.$ Therefore, there is an $H$-ambit homomorphism $s':S(H)\to S(G)$ extending $s$.
		Define $\phi:S(H)\times K\to S(G)$ by $(u,k)\mapsto ks'(u)$. Then $\phi$ is a $G$-ambit homomorphism:
		\begin{align*}
		\phi(g(u,k))=&\phi((\pi(g)u, gk(s(\pi(g)))^{-1}))=gk(s(\pi(g)))^{-1}s'(\pi(g)u)\\
		=&
		gk(s(\pi(g)))^{-1}s(\pi(g))s'(u)=gks'(u)=g\phi(u,k). \tag{$s'$  homomorphism}
		\end{align*}
		As $S(G)$ is the greatest $G$-ambit, we can conclude that $\phi$ is an isomorphism.
		
		Because $M(H)$ is a minimal subflow of $S(H),$ we have that $M(H)\times K$ is a minimal subflow of $S(H)\times K$, and therefore isomorphic to $M(G).$
	\end{proof}
	
	The following result is an immediate application of Theorem \ref{bythm}.
	
	\begin{theorem}
		Let $G\cong K\ltimes H,$ where $K$ is compact. Then $S(G)\cong S(H)\times K$ and $M(G)\cong M(H)\times K.$ 
		\end{theorem}
	
	\section{Concluding remarks and questions}
	A natural question is whether we can prove a version of Theorem \ref{flows} with relaxed requirements. 
	If $K$ is not normal, but there is a uniformly continuous cross section $s:G/K\to G,$ we can prove that the Samuel compactification $S(G/K)$ of the quotient space $G/K$ is homeomorphic to the orbit space $S(G)/K$, and consequently $S(G)$ is homeomorphic to $S(G/K)\times K.$ However,  $G/K$ is not a group, so there is no notion of a $G/K$-flow on $S(G/K)$.
	A natural test problem is whether Corollary \ref{tdlc} holds for all Polish t.d.l.c. groups. Among the easiest examples of Polish t.d.l.c. groups that are not SIN are groups of automorphisms of graphs of finite degree. Answering the following concrete question will likely shed light on this problem. 

	\begin{question}
		What is the universal minimal flow of the group of automorphism of a countable regular $n$-branching tree for $n\geq 3$?
	\end{question}

 Grosser and Moskowitz showed  in \cite{GrM} that a locally compact SIN group $G$  contains a compact normal subgroup $K$ such that $G/K$ is a Lie SIN group. To reduce computation of $M(G)$ to computation of $M(G/K)$ using methods in this paper, we would need to have a uniformly continuous cross section $G/K\to G.$ However, there may not exist continuous cross sections even in compact groups.
 Using this decomposition, they showed 
   that connected locally compact SIN groups are extensions of compact groups by vector groups $\mathbb{R}^n$. More generally, they proved that every locally compact SIN group $G$ is an extension of $\mathbb{R}^n\times K$ with $K$ compact by a discrete group. It is natural to ask whether it is the case that  $M(G)$ can be computed from $M(H),$ where $H$ is an extension of $\mathbb{R}^n$ by a discrete group. Let us remark, that $M(\mathbb{R})$ was computed by Turek in \cite{T3} as a quotient of $M(\mathbb{Z})\times [0,1],$ and his method has recently been generalized to $M(\mathbb{R}^n)$ by Vishnubhotla (see \cite{Vi}). 

\end{document}